\documentclass[12pt,letterpaper]{amsart}
\usepackage{amssymb,mathrsfs,url,setspace,xcolor,mathpazo}
\usepackage[total={6.8in, 9in},centering]{geometry}
\usepackage[colorlinks,allcolors=blue]{hyperref} 
\usepackage{xpatch} 
\xpatchcmd{\proof} 
{\itshape} 
{\bfseries} 
{}
{}
\newtheorem{theorem}{Theorem}
\newtheorem{proposition}{Proposition}
\newtheorem{lemma}{Lemma}
\newtheorem{corollary}{Corollary}
\theoremstyle{remark}
\newtheorem{remark}{Remark}
\newtheorem{example}{Example}
\newtheorem{definition}{Definition}
\newcommand{\C}{\mathbb{C}}
\newcommand{\disk}{\mathbb{D}}
\newcommand{\D}{\Omega}
\newcommand{\J}{\mathcal{J}}
\newcommand{\ep}{\varepsilon}
\newcommand{\Lb}{\overline{L}} 
\newcommand{\Dc}{\overline{\Omega}}
\newcommand{\dbar}{\overline{\partial}}

\newcommand{\wb}{\overline{w}}
\onehalfspace

\title{On compactness and $L^p$-regularity in the $\dbar$-Neumann problem}

\author{S\"{o}nmez \c{S}ahuto\u{g}lu}
\address[S\"{o}nmez \c{S}ahuto\u{g}lu]{University of Toledo, 
Department of Mathematics \& Statistics, Toledo, OH 43606, USA}
\email{Sonmez.Sahutoglu@utoledo.edu}

\author{Yunus E. Zeytuncu}
\address[Yunus E. Zeytuncu]{University of Michigan--Dearborn, 
Department of Mathematics \& Statistics, Dearborn, MI 48128, USA}
\email{zeytuncu@umich.edu}


\thanks{The work of the second author is partially supported by a grant 
	from the Simons Foundation (\#353525).}

\begin{document}

\begin{abstract}
Let $\Omega$ be a $C^4$-smooth bounded pseudoconvex domain 
in $\mathbb{C}^2$. We show that  if the $\dbar$-Neumann operator $N_1$ is 
compact on $L^2_{(0,1)}(\Omega)$ then the embedding operator
$\mathcal{J}:Dom(\overline{\partial})\cap Dom(\overline{\partial}^*) 
\to L^2_{(0,1)}(\Omega)$ is $L^p$-regular for all $2\leq p<\infty$.  
\end{abstract}
\maketitle

\section{Introduction}
Let $\D$ be a bounded pseudoconvex domain in $\C^n$, $1\leq q\leq n$,  
and let $Dom^2(\dbar)$ and $Dom^2(\dbar^*)$ denote the domains of the densely 
defined operators $\dbar$ and $\dbar^*$ in $L^2_{(0,q)}(\D)$, respectively. 
On bounded pseudoconvex domains, H\"{o}rmander in \cite{Hormander65} 
proved  the following basic estimate, 
\[\|f\|_{L^2}\lesssim \|\dbar f\|_{L^2}+\|\dbar^*f\|_{L^2}\] 
for all $(0,q)$-forms $f\in Dom^2(\dbar)\cap Dom^2(\dbar^*)\subset L^2_{(0,q)}(\D)$. 
The sum on the right hand side is called the $L^2$-graph norm of the 
$(0,q)$-form $f$. In other words, the embedding operator
\[\J :Dom^2(\dbar)\cap Dom^2(\dbar^*)\to L^2_{(0,q)}(\D)\] 
is bounded, where the space on the left hand side is endowed with the graph
norm. 

Let $1<p,\widetilde{p}<\infty$ such that $p^{-1}+\widetilde{p}^{-1}=1$. We define 
$Dom^p(\dbar)=\{f\in L^p_{(0,q)}(\D):\dbar f\in  L^p_{(0,q+1)}(\D)\}$. 
We define $Dom^p(\dbar^*)$ as follows: we say 
$f\in Dom^p(\dbar^*)$ if $f\in L^p_{(0,q)}(\D)$ and there exists $C>0$ such that 
\[|\langle f,\dbar g\rangle|\leq C\|g\|_{L^{\widetilde{p}}}\]
for all $g\in L^{\widetilde{p}}_{(0,q-1)}(\D)$ with $\dbar g\in L^{\widetilde{p}}_{(0,q)}(\D)$. 
Finally, we define the space 
\[D^p_{(0,q)}(\D)= Dom^p(\dbar)\cap Dom^p(\dbar^*)\subset L^p_{(0,q)}(\D)\]
and endow it with the $L^p$-graph norm $\|.\|_{G^p}$ defined as 
\[\|f\|_{G^p}= \|\dbar f\|_{L^p}+\|\dbar^* f\|_{L^p}\]
for $f\in D^p_{(0,q)}(\D)$. We note that on bounded pseudoconvex 
domains, $\|.\|_{G^p}$ is a norm because $\dbar f=0$ and $\dbar^* f=0$  
imply that $f=0$ for $1\leq q\leq n$ (see, for example, 
\cite[(4.4.2) in section 4.4]{ChenShawBook}).

\begin{definition}
We say that the operator $\J$ is \textit{$L^p$-regular} on $D^p_{(0,q)}(\D)$ if 
 there exists $C>0$ such that
\[\|\J f\|_{L^p} =\|f\|_{L^p}
\leq C\|f\|_{G^p} =C\left(\|\dbar f\|_{L^p}+\|\dbar^* f\|_{L^p}\right) \] 
for all $f\in D^p_{(0,q)}(\D)$. 
\end{definition} 
That is, whenever $\J:D^p_{(0,q)}(\D) \to   L^p_{(0,q)}(\D)$ is a bounded embedding 
we say that it is $L^p$-regular. In particular, by H\"ormander's basic estimate above, 
$\J$ is $L^2$-regular on bounded pseudoconvex domains. We note that $D^p_{(0,q)}(\D)$ 
is a Banach space (for $1\leq q\leq n$ with  the graph norm $\|.\|_{G^p}$) when $\J$ is 
$L^p$-regular.

The operator $\J$ is related to the $\dbar$-Neumann operator $N$, 
the bounded inverse of the complex Laplacian $\dbar^*\dbar+\dbar\dbar^*$ 
on $L^2_{(0,q)}(\D)$, as $N=\J \J^*$ (see, for example, 
\cite[Proof of Theorem 2.9]{StraubeBook}). 
Hence, $N$ is compact if and only if $\J$ is compact. In this note, we 
show that compactness of $N$ implies $L^p$-regularity of $\J$ for 
$2\leq p<\infty$. We also note that it is not yet clear if $\J $ is $L^p$-regular for 
$1<p<2$ under the compactness assumption. We further note that the question of 
whether  the $\dbar$-Neumann operator or the Bergman projection are bounded 
in  $L^p$-norm whenever $\J$ is compact is open as well.

Although the mapping properties of the canonical operators relate well in the 
$L^2$-Sobolev setting, similar equivalences in the $L^p$ setting are less clear. 
In \cite{BonamiSibony91}, Bonami and Sibony obtained $L^p$ estimates for 
the solutions of $\dbar$-problem under some Sobolev estimates. 
In \cite{HarringtonZeytuncu19} Harrington and Zeytuncu obtained some 
$L^p$ estimates on the canonical operators under the assumption of the existence 
of good weight functions. Both assumptions are more stringent than the compactness 
of $N$ and hence the $L^p$ estimates are more general. Also, recently, Haslinger 
in \cite[Theorem 2.2]{Haslinger16} showed that if $\J$ gains regularity in the 
$L^p$ scale then $N$ is compact. In this paper, we observe a property that is less 
general than the ones in \cite{BonamiSibony91,HarringtonZeytuncu19} under a 
weaker assumption, and that is in the converse direction of the result in \cite{Haslinger16}. 
Namely, in Theorem \ref{Thm1} below, we show that compactness of $N_1$ 
(at the $L^2$ level) induces $L^p$-regularity of $\J$ for $2\leq p<\infty$.

\begin{theorem}\label{Thm1}
Let $\D$ be a $C^4$-smooth bounded pseudoconvex domain in $\C^2$. 
Assume that $N_1$ is compact on $L^2_{(0,1)}(\D)$ (or, equivalently, $\J$ 
is compact on $D^2_{(0,1)}(\D)$). Then   $\J$ is $L^p$-regular 
on $D^p_{(0,1)}(\D)$ for all $2\leq p<\infty$. 
\end{theorem}
We note that the $L^p$ boundedness is not an automatic consequence of 
compactness on $L^2$; as we demonstrate with 
Example \ref{Example2}, in which we present an operator that is compact 
on the $L^2$ space but unbounded on any $L^p$ spaces for $p\neq 2$. 

In the rest of the paper, we use the symbol $x\lesssim y$ to mean that there 
exists $C>0$ such that $x\leq Cy$. Furthermore, when we write a family of 
inequalities depending on a parameter $\ep$ 
\[x\lesssim \ep y,\] 
we mean that there exists $C>0$ that is independent of $\ep$ such 
that $x\leq C\ep y$.

\section{Proof of Theorem \ref{Thm1}} 
{One can prove the following density lemma similarly as in 
\cite[Lemma 4.3.2]{ChenShawBook} (see also \cite[Proposition 2.3]{StraubeBook})
using an $L^p$ version of Friedrichs Lemma 
(see, for instance, \cite[Lemma 3.1]{BlouzaLeDret01}).

\begin{lemma}\label{LemDensity} 
Let $\D$ be a $C^{k+1}$-smooth bounded domain in $\C^n$ $k,1\leq q\leq n$, 
and $1<p<\infty$. Then  $C^k_{(0,q)}(\Dc)\cap Dom(\dbar^*)$ is dense in 
$D^p_{(0,q)}(\D)$ in the graph norm 
$f\to \|f\|_{L^p} +\|\dbar f\|_{L^p}+\|\dbar^*f\|_{L^p}$. 
The statement also holds with $k$ and $k+1$ replaced with $\infty$. 
\end{lemma}

We will need the following lemma which is a corollary of 
\cite[Theorem 1.1]{JerisonKenig95}.
\begin{lemma}[Jerison-Kenig]\label{LemNegative}
Let $\D$ be a $C^1$-smooth bounded domain in $\mathbb{R}^n$ 
and $1<p<\infty$. Then there exists $C>0$ such that 
\begin{align}\label{EqnIso}
\|u\|_{W^{1,p}} \leq C \|\Delta u\|_{W^{-1,p}} 
\end{align}
for all $u\in W^{1,p}_0(\D)$.
\end{lemma}

Using the lemmas above together with the proof of \cite[Lemma 2.12]{StraubeBook} 
one can prove the following estimate on the normal component of forms. We note 
that, in the lemma below, $f_{norm}$ denotes the normal component of $f$ 
(see (2.86) in \cite{StraubeBook}). 

\begin{lemma}\label{LemNormal}
Let $\D$ be a $C^4$-smooth bounded pseudoconvex domain in $\C^n,1\leq q\leq n$, 
and $1<p<\infty$. There exists $C>0$ such that if $f\in D^p_{(0,q)}(\D)$ then 
$f_{norm}\in W^{1,p}_{0,(0,q-1)}(\D)$ and 
\[ \|f_{norm}\|_{W^{1,p}}\leq C\left(\|\dbar f\|_{L^p}+\|\dbar^* f\|_{L^p}+\|f\|_{L^p}\right).\]
\end{lemma}

We will use Lemma \ref{LemNegative} to also prove the following $L^p$ version of 
\cite[Proposition 5.1.1]{ChenShawBook}. 

\begin{proposition}\label{PropMult}
Let $\D$ be a $C^2$-smooth bounded domain in $\C^n, 1<p<\infty,1\leq q\leq n,$ 
and $\phi\in C^1(\Dc)$ such that $\phi=0$ on $b\D$. Then there exists $C>0$ such that 
\begin{align*}
\|\phi f\|_{W^{1,p}} \leq C(\|\dbar f\|_{L^p}+\|\dbar^* f\|_{L^p}+\|f\|_{L^p}) 
\end{align*}
for $f\in D^p_{(0,q)}(\D)$. 
\end{proposition} 

\begin{proof} 
First we assume that $g\in  D^p_{(0,q)}(\D)$ with coefficient functions in  $W^{1,p}_0(\D)$. 
Then we have 
\begin{align*}
\|g\|_{W^{1,p}} \lesssim \|\Delta g\|_{W^{-1,p}}\lesssim \|\dbar g\|_{L^p}+\|\dbar^*g\|_{L^p}.
\end{align*}
Then we substitute $g=\phi f$ for $f\in C^1_{(0,q)}(\Dc)\cap Dom(\dbar^*)$ in the 
inequality above
\begin{align*}
\|\phi f\|_{W^{1,p}} \lesssim \|\dbar (\phi f)\|_{L^p}+\|\dbar^* (\phi f)\|_{L^p}
\lesssim \|\dbar f\|_{L^p}+\|\dbar^* f\|_{L^p}+\|f\|_{L^p}.
\end{align*}
Then we use Lemma \ref{LemDensity} to conclude the proof. 
\end{proof}

The interpolation inequality for Sobolev spaces together with 
Proposition \ref{PropMult} imply the following corollary.

\begin{corollary}\label{CorCompMult}
Let $\D$ be a $C^2$-smooth bounded domain in $\C^n, 1<p<\infty, 1\leq q\leq n,$ 
and $\phi\in C(\Dc)$ such that $\phi=0$ on $b\D$. Then the multiplication operator 
$M_{\phi}:D^p_{(0,q)}(\D)\to L^p_{(0,q)}(\D)$ is compact. 
\end{corollary}
In other words, in the terminology of \cite{CelikStraube09}, continuous 
functions on $\Dc$ that vanish on the boundary are compactness multipliers. 

We note that even though \cite[Lemma 4.3]{StraubeBook}
is stated for Hilbert spaces the proof works for Banach spaces 
as well. In the proof of i) of Lemma \ref{LemCompEst} below 
one uses the facts that on reflective Banach spaces bounded 
sequences have weakly convergent subsequences 
(see \cite[Theorem 1 on pg 126]{YosidaBook}) as well as 
compact operators map weakly convergent sequences to convergent 
sequences. Therefore, proof of \cite[Lemma 4.3]{StraubeBook} 
(see also exercise 6.13 in \cite{BrezisBook}) implies the 
following lemma.
\begin{lemma}\label{LemCompEst}
Let $T:X\to Y$ be a bounded linear map where $X$ is a normed linear space 
and $Y$ is a Banach space. 
\begin{itemize}
\item[i.] 
Assume that for all $\ep>0$ there exist a Banach space 
$Z_{\ep}$ and a compact linear map $K_{\ep}:X\to Z_{\ep}$ such that 
\[\|Tx\|_Y\leq \ep\|x\|_X+\|K_{\ep} x\|_{Z_{\ep}}\]
for all $x\in X$. Then $T$ is compact.
\item [ii.] Assume that $X$ is reflexive Banach space, $T$ is compact and  
$K:X\to Z$ is an injective bounded linear map. Then for any $\ep>0$ there 
exists $C_{\ep}>0$ such that 
\[\|Tx\|_Y\leq \ep\|x\|_X+C_{\ep}\|K x\|_Z\]
for all $x\in X$.
\end{itemize}
\end{lemma}

\begin{proof}[Proof of Theorem \ref{Thm1}]
We define $K:D^p_{(0,1)}(\D)\to L^p_{(0,1)}(\D)$ as $K f=\rho f$ where 
$\rho(z)=dist(z,b\D)$ is the distance of $z$ to the boundary of $\D$. Then 
Corollary \ref{CorCompMult} implies that $K$ is compact for all $1<p<\infty$. 
We note that $K$ is an injection as well. 

Since $\D$ is a bounded pseudoconvex domain, $D^2_{(0,1)}(\D)$ is a Hilbert space. 
Then we use ii. in Lemma \ref{LemCompEst} to get the following estimates: 
for any $\ep>0$ there exists $C_{\ep}>0$ such that 
\begin{align*} 
\|f\|_{L^2}\leq \ep(\|\dbar f\|_{L^2}+\|\dbar^* f\|_{L^2})+C_{\ep}\|Kf\|_{L^2}
\end{align*}
for all $f\in D^2_{(0,1)}(\D)$.

First we show how to get $L^4$-regularity. Let 
$F= f_1\wb_1+f_2\wb_2$ be in $D^4_{(0,1)}(\D)\subset L^4_{(0,1)}(\D)$  
such that $f_2$ is the normal component. 
Because of Lemma \ref{LemDensity}, without loss of generality, 
we may assume that $f_1$ and $f_2$ are $C^3$-smooth on $\Dc$.
We denote $F_2=f_1^2\wb_1+f_2^2\wb_2$. Since $f_2$ vanishes 
on the boundary we have $F_2\in D^2_{(0,1)}(\D)$. Then 
$\|F\|_{L^4}^4\approx \|F_2\|_{L^2}^2<\infty$ and 
\begin{align*}
\|F_2\|_{L^2}\leq & \ep (\|\dbar F_2\|_{L^2}+ \|\dbar^*F_2\|_{L^2})+C_{\ep}\|KF_2\|_{L^2} \\
\lesssim &\ep\left(\|f_1\Lb_2f_1-f_2\Lb_1f_2\|_{L^2}
+\|f_1L_1f_1+f_2L_2f_2\|_{L^2}+\|F_2\|_{L^2}\right)\\
&+C_{\ep}\|KF_2\|_{L^2}\\
\lesssim &\ep\left(\|f_1\Lb_2f_1-f_1\Lb_1f_2\|_{L^2}
+\|f_1\Lb_1f_2-f_2\Lb_1f_2\|_{L^2}\right) \\
& +\ep \left(\|f_1L_1f_1+f_1L_2f_2\|_{L^2} 
+\|f_1L_2f_2-f_2L_2f_2\|_{L^2}+\|F_2\|_{L^2}\right)\\
&+C_{\ep}\|KF_2\|_{L^2}
\end{align*}
By absorbing the terms that are multiple of $\|F_2\|_{L^2}$ 
into the left hand side we get
\begin{align}\label{Eqn3}
\|F_2\|_{L^2}\lesssim& \ep\left(\|f_1\dbar F\|_{L^2} +\|(f_1-f_2)\Lb_1f_2\|_{L^2}
+ \|f_1\dbar^* F\|_{L^2} +\|(f_1-f_2)L_2f_2\|_{L^2}\right)\\
\nonumber &+C_{\ep}\|KF_2\|_{L^2}.
\end{align}
Using the facts that $\|F_2\|_{L^2}\approx \|F\|^2_{L^4}<\infty$, $KF_2=\rho F_2$,  
and the Cauchy-Schwarz inequality we get 
\begin{align*}
\|F\|^2_{L^4}\lesssim &
\ep\|F\|_{L^4}\left(\|\dbar F\|_{L^4}+\|\dbar^* F\|_{L^4}+\|f_2\|_{W^{1,4}}\right) 
+C_{\ep}\|\rho^{1/2}F\|^2_{L^4}\\
\lesssim & \ep\|F\|_{L^4}\left(\|\dbar F\|_{L^4} +\|\dbar^* F\|_{L^4}+\|f_2\|_{W^{1,4}}\right)
+C_{\ep}\|F\|_{L^4}\|\rho F\|_{L^4}. 
\end{align*} 
Using the inequality $2|xy|\leq |x|^2+|y|^2$ on right hand 
side we can absorb $\|F\|_{L^4}$ into the left hand side and get 
($C_{\ep}$ below is different from its previous values, but it still depends on $\ep$ only)
\begin{align}\label{Eqn1}
\|F\|^2_{L^4}\lesssim 
\ep\left(\|\dbar F\|^2_{L^4} +\|\dbar^* F\|^2_{L^4}+\|f_2\|^2_{W^{1,4}}\right) 
+C_{\ep}\|\rho F\|^2_{L^4}. 
\end{align}

Now we will concentrate on $\|f_2\|_{W^{1,4}}$. Using Lemma \ref{LemNegative} 
and Lemma \ref{LemNormal} we get 
\begin{align*}
\|f_2\|^2_{W^{1,4}} \lesssim \|\dbar F\|^2_{L^4}+\|\dbar^* F\|^2_{L^4} +\|F\|^2_{L^4}.
\end{align*}
Then the inequality \eqref{Eqn1} turns into
\begin{align*} 
\|F\|^2_{L^4}\lesssim 
\ep\left(\|\dbar F\|^2_{L^4}+\|\dbar^* F\|^2_{L^4}\right) +C_{\ep}\|\rho F\|^2_{L^4}. 
\end{align*} 
That is, we showed that for $\ep>0$ given there exists $C_{\ep}>0$ such that 
\begin{align}\label{Eqn2}
\|\J F\|_{L^4} \leq \ep\left(\|\dbar F\|_{L^4}+\|\dbar^* F\|_{L^4}\right) +C_{\ep}\|KF\|_{L^4} 
\end{align} 
for $F\in D^4_{(0,1)}(\D)$. Therefore, $\J:D^4_{(0,1)}(\D)\to L^4_{(0,1)}(\D)$ 
is a compact operator. Furthermore, since $\J$ is $L^4$-regular, 
one can show that  $D^4_{(0,1)}(\D)$ is a Banach space.

In a similar fashion, we use estimates \eqref{Eqn2} to show that for 
every $\ep>0$ there exists $C_{\ep}>0$ such that  
\begin{align*}
\|\J F\|_{L^8} \leq  
\ep\left(\|\dbar F\|_{L^8}+\|\dbar^* F\|_{L^8}\right) +C_{\ep}\|KF\|_{L^8}
\end{align*} 
for $F\in D^8_{(0,1)}(\D)$. That is,  $\J:D^8_{(0,1)}(\D)\to L^8_{(0,1)}(\D)$ 
is a compact linear map (by Lemma \ref{LemCompEst}) and $D^8_{(0,1)}(\D)$ 
is a Banach space. Inductively, we show that  
$\J:D^{2^p}_{(0,1)}(\D)\to L^{2^p}_{(0,1)}(\D)$ is a compact 
linear map and $D^{2^p}_{(0,1)}(\D)$ is a Banach space for $p\in\mathbb{Z}^+$. 

Note that for any  $p\in\mathbb{Z}^+$, we have 
$D^{2^p}_{(0,1)}(\D)\cap D^{2^{p+1}}_{(0,1)}(\D)=D^{2^{p+1}}_{(0,1)}(\D)$ 
and $D^{2^p}_{(0,1)}(\D)+ D^{2^{p+1}}_{(0,1)}(\D)\subset D^{2^{p}}_{(0,1)}(\D)$. 
In other words, for $2^p<q<2^{p+1}$ we get 
\[D^{2^p}_{(0,1)}(\D)\subset D^{q}_{(0,1)}(\D)\subset D^{2^{p+1}}_{(0,1)}(\D)\]
and since the graph norm is the sum of $L^p$ norms we conclude that 
$D^{q}_{(0,1)}(\D)$  is an intermediate space (\cite[Definition 2.4.1]{Bergh}) 
for twoBanach spaces $D^{2^p}_{(0,1)}(\D)$ and $D^{2^{p+1}}_{(0,1)}(\D)$. 
Now, by the complex interpolation theorem (\cite[Chapter 4]{Bergh}) 
we conclude that $\J:D^p_{(0,1)}(\D)\to L^p_{(0,1)}(\D)$ is 
$L^p$-regular and $D^p_{(0,1)}(\D)$ is a Banach space for all $2\leq p<\infty$. 
\end{proof}

\begin{remark}
The proof of Theorem \ref{Thm1} shows that we have the same result 
for $(p,n-1)$-forms on $C^4$-smooth bounded pseudoconvex domains in $\C^n$. 
\end{remark}

We note that $Ker(\dbar)$ and $ A^2(\D)^{\perp}$ denote the set of 
$\dbar$-closed forms and the orthogonal complement of the 
Bergman space $A^2(\D)\subset L^2(\D)$, respectively.

\begin{proposition}\label{Prop1}
Let $\D$ be a  bounded pseudoconvex domain in $\C^n,n\geq 2,$ and $1<p\leq 2$. 
Assume that $\J$ is $L^p$-regular on $D^p_{(0,1)}(\D)$. Then the following operators 
are bounded 
\begin{itemize}
\item [i.] $\dbar^*N_2:L^p_{(0,2)}(\D)\cap L^2_{(0,2)}(\D)\cap Ker(\dbar) \to L^p_{(0,1)}(\D)$, 
\item[ii.] $\dbar N_0:L^p(\D)\cap L^2(\D)\cap  A^2(\D)^{\perp} \to L^p_{(0,1)}(\D)$.
\end{itemize}
\end{proposition}

\begin{proof}
Since $\J$ is $L^p$-regular and there exists $C>0$ such that
\begin{align}\label{EqnBasicEst}
	\|f\|_{L^p} \leq C\left(\|\dbar f\|_{L^p}+\|\dbar^* f\|_{L^p}\right)
\end{align} 
for $f\in D^p_{(0,1)}(\D)$. Note that $\dbar N_0g\in Dom^2(\dbar^*)\subset Dom^p(\dbar^*)$ 
for $g\in  L^p(\D)\cap L^2(\D)\cap A^2(\D)^{\perp}$ and $p\leq 2$. Then applying  
the estimate \eqref{EqnBasicEst} to $\dbar N_0g$ we get 
\[\|\dbar N_0g\|_{L^p} \leq C\|\dbar^*\dbar N_0 g\|_{L^p}=C\|g\|_{L^p}\]
for $g\in  L^2(\D)\cap A^2(\D)^{\perp}$. 

Similarly, if we apply \eqref{EqnBasicEst} to $\dbar^*N_2h$ with 
$h\in L^p_{(0,2)}(\D)\cap L^2_{(0,2)}(\D)\cap  Ker(\dbar)$ we get 
\[\|\dbar^*N_2h\|_{L^p} \leq C\|\dbar\dbar^* N_2 h\|_{L^p}=C\|h\|_{L^p}.\] 
Hence the proof of the proposition is complete. 
\end{proof} 

The following example shows that the $L^p$ boundedness of an operator $T$ 
is not an automatic  consequence of the compactness of $T$ on $L^2$.  
\begin{example} \label{Example2}
Set 
\[\phi(z)=\exp\left(\frac{-1}{1-|z|}\right)\] 
and consider the weighted Bergman space $A^2(\disk, \phi)$ on the unit disc. 
The weighted Bergman projection $\mathbf{B}_{\phi}$ is studied in 
\cite{Dostanic,DostanicRev,ZeyTran}, and it was noted that $\mathbf{B}_{\phi}$ 
is unbounded on $L^p(\mathbb{D},\phi)$ for any $p\not=2$. 

We define an operator $T$ on $L^2(\mathbb{D},\phi)$ by
\begin{align*}
T&:L^2(\mathbb{D},\phi)\to L^2(\mathbb{D},\phi)\\
Tf(z)&=\mathbf{B}_{\phi}(f)(z)(1-|z|^2)^{2}.
\end{align*}
The operator $T$ is bounded, linear and self-adjoint. Furthermore, 
we show  that $T$ is compact on $L^2(\mathbb{D},\phi)$ yet it is 
unbounded on $L^p(\mathbb{D},\phi)$  for any $p\not=2$.

First we show that $T$ is compact. For $\ep>0$ there exists a compact set 
$K_{\ep}\subset \mathbb{D}$ such that $(1-|z|^2)^2<\ep$ on $\mathbb{D}\setminus K$. 
We define $S_{\ep}f=\chi_{K_{\ep}}Tf$ where $\chi_{K_{\ep}}$ is the characteristic 
function of $K_{\ep}$. Montel's theorem implies that $S_{\ep}$ is  compact. 
\[\|Tf\|^2=\|Tf\|^2_{L^2(\mathbb{D}\setminus K,\phi)}+  \|Tf\|^2_{L^2(K,\phi)} 
\leq \ep\|\mathbf{B}_{\phi}f\|^2+\|S_{\ep}f\|^2\leq\ep\|f\|^2+\|S_{\ep}f\|^2.\] 
That is, $T$ satisfies compactness estimates and by Lemma \ref{LemCompEst}  
it is a compact operator on $L^2(\mathbb{D},\phi)$ (see also 
 \cite[Proposition V.2.3]{D`AngeloIneqBook} or \cite[Lemma 4.3]{StraubeBook}).

Next we show that $T$ is unbounded on $L^p(\mathbb{D},\phi)$ for any 
$p\neq 2$. Let $0< p<2$ and 
\[f_n(z)=z^{kn}\overline{z}^n\] 
where $k$ is a positive integer to be determined later. Then one can compute that 
\[Tf_n(z)=a_nz^{kn-n}(1-|z|^2)^2\] 
where 
\[a_n=\frac{\int_{\mathbb{D}}|z|^{2kn}\phi(z)dA(z)}{\int_{\mathbb{D}}|z|^{2(k-1)n} 
\phi(z)dA(z)}.\]
Furthermore, 
\begin{align*}
\frac{\|Tf_n\|_p^p}{\|f_n\|_p^p} 
=\left(\frac{\int_{\mathbb{D}}|z|^{2kn}\phi(z)dA(z)}{\int_{\mathbb{D}}|z|^{2(k-1)n}\phi(z)dA(z)}\right)^p 
\frac{\int_{\mathbb{D}}|z|^{pkn-pn}(1-|z|^2)^{2p}\phi(z)dA(z)}{\int_{\mathbb{D}}|z|^{pkn+pn}\phi(z)dA(z)}.
\end{align*}
We need the following asymptotic \cite[Lemma 1]{DostanicRev}
\[\int_{\mathbb{D}}|z|^{t}(1-|z|^2)^{2s}\phi(z)dA(z) \sim (t+1)^{\frac{-4s-3}{4}}\exp(-2\sqrt{t+1}) \]
as $t\to \infty$. 

We have the following asymptotic computations 
\begin{align*}
\frac{\|Tf_n\|_p^p}{\|f_n\|_p^p}
&=\left(\frac{\int_{\mathbb{D}}|z|^{2kn}\phi(z)dA(z)}{\int_{\mathbb{D}}|z|^{2(k-1)n} \phi(z)dA(z)}\right)^p 
\frac{\int_{\mathbb{D}}|z|^{pkn-pn}(1-|z|^2)^{2p}\phi(z)dA(z)}{\int_{\mathbb{D}}|z|^{pkn+pn}\phi(z)dA(z)}\\
&\sim \frac{(2kn+1)^{-3p/4}\exp(-2p\sqrt{2kn+1})}{(2kn-2n+1)^{-3p/4}\exp(-2p\sqrt{2kn-2n+1})}\\
&\times 
\frac{(pkn-pn+1)^{(-4p-3)/4}\exp(-2\sqrt{pkn-pn+1})}{(pkn+pn+1)^{-3/4}
\exp(-2\sqrt{pkn+pn+1})}\\
&\sim C_{k,p}n^{-p}\exp(2D_{k,p,n})
\end{align*}
as $n\to\infty$ where 
\[C_{k,p}=\frac{(k+1)^{3/4}}{p^pk^{3p/4}(k-1)^{(3+p)/4}}\] 
and 
\begin{align*}
D_{k,p,n}=&-p\sqrt{2kn+1}+p\sqrt{2kn-2n+1}-\sqrt{pkn-pn+1}+\sqrt{pkn+pn+1}\\
=&\frac{-2pn}{\sqrt{2kn+1}+\sqrt{2kn-2n+1}}+\frac{2pn}{\sqrt{pkn-pn+1}+\sqrt{pkn+pn+1}}\\
\geq &\frac{pn}{\sqrt{pkn+pn+1}}-\frac{pn}{\sqrt{2kn-2n+1}}\\
\geq &p\sqrt{n}\left( \frac{1}{\sqrt{pk+p+1}}- \frac{1}{\sqrt{2k-2}} \right). 
\end{align*}
Then one can show that for large $k$ we have 
\[\frac{1}{\sqrt{pk+p+1}}- \frac{1}{\sqrt{2k-2}} 
\geq \frac{1}{2}\left(\frac{1}{\sqrt{p}}-\frac{1}{\sqrt{2}}\right)>0.\]
Therefore, for large $k$ we have 
\[D_{k,p,n}\geq \frac{\sqrt{n}p}{2}\left(\frac{1}{\sqrt{p}}-\frac{1}{\sqrt{2}}\right).\]
Therefore, for $k$ large enough we have $C_{k,p}n^{-p}\exp(2D_{k,p,n})\to \infty$ as $n\to\infty$. 
Then we conclude that $\frac{\|Tf_n\|_p}{\|f_n\|_p}\to \infty$ as $n\to\infty$. 
Hence $T$ is  not bounded on $L^p(\disk,\phi)$ for any $p<2$.
Furthermore, the fact that $T$ is self-adjoint implies that $T$ is  unbounded 
on $L^p(\disk,\phi)$ for any $p\neq 2$.
\end{example} 

\section{Acknowledgment}
We would like to thank Emil Straube for reading a previous version of the 
manuscript and providing valuable feedback. We would like to thank the anonymous referees for constructive remarks in an earlier version of this paper.


\end{document}